\documentclass[12pt]{article}

\usepackage{amsmath,amssymb,amsfonts}
\usepackage{amsthm}
\usepackage{enumerate}
\usepackage[utf8]{inputenc}
\usepackage{mathrsfs}

\allowdisplaybreaks%

{\theoremstyle{plain}%
 \newtheorem{theorem}{Theorem}
 
 \newtheorem{lemma}{Lemma}%
}
{\theoremstyle{remark}
\newtheorem{remark}{Remark}
}
{\theoremstyle{definition}

}

\begin{document}

\begin{center}
 {\Large On divisor sums due to Erd\H os and Ramanujan}

 \ 

{\sc John M. Campbell}

\vspace{0.1in}

{\footnotesize Department of Mathematics and Statistics}

{\footnotesize Dalhousie University}

{\footnotesize Halifax, NS B3H 4R2}

{\footnotesize Canada}

\vspace{0.1in}

{\footnotesize {\tt jh241966@dal.ca}}

\vspace{0.1in}

\end{center}

\begin{abstract}
 Let $d(n)$ denote the number of divisors of a positive integer $n$. A classical problem in analytic number theory is given by the 
 asymptotic behavior of the divisor sum $\sum_{n \leq x} \frac{1}{d(n)}$, with Ramanujan having introduced an asymptotic formula for 
 this sum with an explicit evaluation for the constant $A_1$ for the leading term $A_1 \frac{x}{\sqrt{\log x}}$. Gabdullin et al.\ recently 
 considered a hybrid of this problem and the Titchmarsh divisor problem concerning $\sum_{p\leq x} d(p-1)$, proving that 
 $\sum_{p\leq x} \frac{1}{d(p-1)} \asymp \frac{x}{(\log x)^{3/2}}$. This result and Erd\H os's asymptotic formula 
 $ \sum_{n \leq x} d(d(n)) \sim c_1 \, x \log \log x $ for a constant $c_1 \in (0, \infty)$ lead us to consider the hybrid $\sum_{n \leq x} 
 \frac{1}{d(d(n))}$ of the Erd\H os and Ramanujan divisor sums. In this paper, we prove that there exists a constant $c_3 \in (0, \infty)$ 
 such that $$\sum_{n \leq x} \frac{1}{d(d(n))} \sim c_3 \frac{x}{ \log \log x}, $$ through a combined application of Golomb's estimate 
 for powerful numbers, Tur\'{a}n's quantitative form of the Hardy--Ramanujan theorem, and a fourth-moment estimate for the number 
 of distinct prime divisors due to Granville and Soundararajan. Moreover, an explicit formula for the constant $c_3$ is given. 
\end{abstract}

\vspace{0.1in}


\vspace{0.1in}


\section{Introduction}
 The divisor function $d(n) = \sum_{d \mid n} 1 $ is among the most fundamental arithmetic functions. The iteration $d(d(n))$ is closely 
 related to Ramanujan's seminal work on highly composite numbers \cite{Ramanujan1915comp}, and this motivated 
 Erd\H{o}s~\cite{Erdos1968} to investigate the behavior of sums of the form $ \sum_{n \leq x} d(d(n))$.    In this direction, Erd\H os proved 
 the remarkable result whereby: There exists a constant $c_1$ such that $ 0 < c_1 < \infty$ and such that 
\begin{equation}\label{displayErdos}
 \sum_{n \leq x} d(d(n)) \sim c_1 \, x \log \log x. 
\end{equation}
 This may be seen in contrast to Dirichlet's formula 
\begin{equation}\label{displayDirichlet}
 \sum_{n\leq x} d(n) = x \log x + (2 \gamma -1) x + O(\sqrt{x}), 
\end{equation}
 for the Euler--Mascheroni constant $\gamma = \lim_{n \to \infty} \big( 1 + \frac{1}{2} 
 + \cdots + \frac{1}{n} - \log n \big)$, in the sense that Dirichlet’s hyperbola method 
 used in the classical derivation of \eqref{displayDirichlet} cannot be applied directly to the sum in \eqref{displayErdos}. 
 Since $d(d(n))$ is not multiplicative, in contrast to $d(n)$ being multiplicative, this may be seen 
 as complicating the analysis of the behavior of $d(d(n))$. 
 Indeed, the difficulties surrounding the behavior of the self-composition $d(d(n))$ may be seen by 
 the solution due to Buttkewitz et al.\ in 2012 \cite{ButtkewitzElsholtzFordSchlagePuchta2012} of a problem due to Ramanujan 
 \cite{Ramanujan1915comp} concerning the maximal order, asymptotically, of $\log d(d(n))$. The Erd\H os estimate in 
 \eqref{displayErdos} may also be seen in relation to a number of subsequent research works improving upon or otherwise extending 
 this estimate~\cite{ErdosKatai1969,Heppner1974,Katai1969,Rieger1972}. 

 In 1915, Ramanujan \cite{Ramanujan1915analytic} introduced asymptotic evaluations for a number of divisor sums, including 
\begin{multline}\label{displayRamanujan}
 \sum_{n \leq x} \frac{1}{d(n)} = \\ 
 x \left( \frac{A_1}{ (\log x)^{\frac{1}{2}} } 
 + \frac{A_2}{ (\log x)^{\frac{3}{2}} } 
 + \cdots + \frac{A_r}{(\log x)^{r - \frac{1}{2}}} + O\left( \frac{1}{(\log x)^{r + \frac{1}{2}}} \right) \right), 
\end{multline}
 for $$ A_{1} = \frac{1}{\sqrt{\pi}} \prod_p \sqrt{p^2 - p} \log\left( \frac{p}{p-1} \right) $$ and for unspecified constants $A_{2}$, $A_{3}$, 
 $\ldots$, $A_{r}$, with a full proof of Ramanujan's estimate having been introduced by Wilson in 1923 \cite{Wilson1923}. If we compare 
 the Erd\H os divisor sum $ \sum_{n \leq x} d(d(n))$ to the Ramanujan divisor sum $ \sum_{n \leq x} \frac{1}{d(n)}$, this 
 and the work of Gabdullin et al.\ \cite{GabdullinKonyaginIudelevich2023} reviewed below 
 raise questions concerning the 
 hybrid divisor sum $\sum_{n \leq x} \frac{1}{d(d(n))}$, which provides the main object of study in this paper. 

\subsection{A hybrid Titchmarsh--Ramanujan divisor sum}
 The \emph{Titchmarsh divisor problem} or \emph{Titchmarsh--Linnik divisor problem} \cite{Halberstam1967} refers to the problem given 
 by determining the asymptotic behavior of the divisor sum $\sum_{p \leq x} d(p-1)$, with the estimate 
\begin{equation}\label{firstTitchmarsh}
 \sum_{p \leq x} d(p-1) = O(x) 
\end{equation}
 having been introduced and proved by Titchmarsh in 1930 \cite{Titchmarsh1930}. Letting $\zeta(s)$ denote the Riemann zeta function, 
 a remarkable improvement to \eqref{firstTitchmarsh} was introduced and proved in 1961 by Linnik \cite{Linnik1961}, who proved, via a 
 dispersion method, that 
\begin{equation}\label{displayLinnik}
 \sum_{p \leq x} d(p-1) \sim \frac{ \zeta(2) \zeta(3) }{ \zeta(6) } x 
\end{equation}
 with an error term $\ll \frac{x}{ (\log x)^{\alpha} }$ for $0 < \alpha < 1$, and one may compare this against simplified proofs of this 
 estimate due to Rodriquez \cite{Rodriquez1965} and Halberstam \cite{Halberstam1967}, along with the improvement upon Linnik's 
 error term due to Bombieri et al.~\cite{BombieriFriedlanderIwaniec1986}. 

 As a natural hybrid of the Ramanujan divisor sum $ \sum_{n \leq x} \frac{1}{d(n)}$ and the Titchmarsh divisor sum shown in 
 \eqref{firstTitchmarsh} and \eqref{displayLinnik}, 
 Gabdullin et al.\ \cite{GabdullinKonyaginIudelevich2023}
 studied the divisor sum 
 $\sum_{p\leq x} \frac{1}{d(p-1)}$ and succeeded in proving 
 the order of magnitude evaluation 
\begin{equation}\label{magnitudeGabdullin}
 \sum_{p\leq x} \frac{1}{d(p-1)} \asymp \frac{x}{(\log x)^{3/2}}, 
\end{equation}
 thus improving upon a previous result due to Iudelevich \cite{Iudelevich2022}. Gabdullin et al.\ also conjectured that there exists a 
 constant $c_2$ such that $0 < c_2 < \infty$ and such that 
\begin{equation}\label{Gabdullinconj}
 \sum_{p\leq x} \frac{1}{d(p-1)} \sim c_{2} \frac{x}{(\log x)^{3/2}}
\end{equation}
 and noted that this is likely a difficult conjecture. 

\subsection{A hybrid Erd\H{o}s--Ramanujan divisor sum}
 The problem of estimating the divisor sum indicated in \eqref{magnitudeGabdullin} and \eqref{Gabdullinconj} is referred to as 
 \emph{Karatsuba’s divisor problem} \cite{GabdullinKonyaginIudelevich2023,Iudelevich2022}. Since this concerns the hybrid $\sum_{p 
 \leq x} \frac{1}{d(p-1)}$ of the Ramanujan divisor sum $\sum_{n\leq x} \frac{1}{d(n)}$ and the Titchmarsh divisor sum $\sum_{p \leq 
 x} d(p - 1)$, this motivates the study of the corresponding hybrid divisor sum 
 $\sum_{n \leq x} \frac{1}{d(d(n))}$ given by 
 a ``fusion'' of the same Ramanujan divisor sum 
 and the Erd\H os divisor sum $\sum_{n \leq x} d(d(n))$. 
 This is further motivated by the work of 
 K\'atai \cite{Katai1969} 
 on the divisor sum $\sum_{p \leq x} d(d(p-1))$, 
 which may be seen as a hybrid of the 
 Titchmarsh divisor sum in \eqref{firstTitchmarsh}
 and the Erd\H{o}s divisor sum in \eqref{displayErdos}. 

We first prove the order-of-magnitude estimate
$$ \sum_{n \leq x} \frac{1}{d(d(n))}
\asymp \frac{x}{\log \log x}, $$
analogous to the Gabdullin--Konyagin--Iudelevich estimate in
\eqref{magnitudeGabdullin} and the Erd\H{o}s asymptotic estimate
in \eqref{displayErdos}. We then refine this estimate by proving
that there exists a constant $c_3$ such that $0 < c_3 < \infty$
and such that
\begin{equation}\label{newconj}
\sum_{n \leq x} \frac{1}{d(d(n))}
\sim c_3 \frac{x}{\log \log x}.
\end{equation}
 Moreover, we prove an explicit formula for the constant $c_3$ involved in \eqref{newconj}. 

 Informally, when dealing with sums of reciprocals of divisors, as opposed to ``non-reciprocal'' divisor sums, applying the reciprocal 
 operation on divisors can be thought of as accentuating the irregular behavior of $d(n)$ or its subsequences. 
 Informally, this may be seen by the contrast between Ramanujan's formula for $\sum_{n \leq x} \frac{1}{d(n)}$
 and Dirichlet's formula for $\sum_{n \leq x} d(n)$, 
 and this is further 
 evidenced by 
 Gabdullin et al.\ \cite{GabdullinKonyaginIudelevich2023} commenting on the difficulty of the conjecture in \eqref{Gabdullinconj}, in 
 contrast to the original Titchmarsh divisor problem. 

 The research contribution due to Ramanujan \cite{Ramanujan1915analytic} that introduced the divisor sum relation in 
 \eqref{displayRamanujan} may be seen as seminal in the history of divisor sums. 
 In this same contribution, Ramanujan introduced
 an asymptotic evaluation for 
 $ \sum_{n \leq x} d^{2}(n)$, 
 which has inspired a number of notable 
 advances on divisor sums of this form \cite{CullyHugillTrudgian2021,JiaSankaranarayanan2014,MaierSankaranarayanan2005,RamachandraSankaranarayanan2003}. 

\section{Required preliminaries}\label{secprelim}
 For real-valued functions $f$ and $g$ (defined, say, on a subset of $\mathbb{R}_{\geq 0}$), if there exists a constant $C > 0$ and a value 
 $x_0$ such that $ |f(x)| \leq C g(x) $ for all $x \geq x_0$, then we write $f(x) \ll g(x)$. Observe that this is equivalent to $f(x) = O(g(x))$. 
 In a similar spirit, if both $f(x) \ll g(x)$ and $g(x) \ll f(x)$, then it is standard to write $f(x) \asymp g(x)$. 

 The prime omega function $\omega(n)$ gives the number of distinct prime factors of $n \geq 2$, with $\omega(1) = 0$. Similarly, we 
 have that $\Omega(n)$
 gives the total number of prime factors of $n$, counting multiplicities. 
 We then let \(a(n)\) denote the product of the primes \(p\) such that \(p\Vert n\), 
 i.e., the \emph{squarefree} part of \(n\) supported on primes dividing \(n\) exactly once. 
 By letting the prime factorization of $n$ be written as
 $n = \prod_{i} p_{i}^{e_i}$, the \emph{powerful part} of $n$
 may be defined so that 
\begin{equation}\label{displaybn}
 b(n) = \prod_{e_{i} >1} p_{i}^{e_{i}}. 
\end{equation} 
 The definitions of $d(n)$, $\omega(n)$, $a(n)$, and $b(n)$ together give us that 
\begin{equation}\label{keyErdos}
 d(n) = 2^{\omega(a(n))} d(b(n)), 
\end{equation}
 and \eqref{keyErdos} provides a key in the approach used by Erd\H os~\cite{Erdos1968} in the derivation of the 
 asymptotic relation on display in \eqref{displayErdos}. 

\subsection{The Abel summation formula}
 Since this paper concerns average values of arithmetic functions, it is natural that summation techniques often employed in analytic 
 number theory would arise. For a full preliminary treatment of such summation tools, we refer to the appropriate section of De Koninck 
 and Luca's text on analytic number theory \cite[\S1]{DeKoninckLuca2012}, and the following formulation of the Abel summation lemma is 
 required for our purposes. 
 
 \ 

\noindent {\bf (Abel summation formula):} Let $( a_{n} )_{n \geq 1}$ be a sequence of complex numbers and let $f : [1, +\infty) \to 
 \mathbb{C}$ have continuous derivatives for $x \geq 1$. For each real number $x \geq 1$, let $$ A(x) = \sum_{n \leq x} a_n. $$ Then 
 \begin{equation}\label{eq1p4} 
 \sum_{n \leq x} a_{n} f(n) = A(x) f(x) - \int_{1}^{x} A(t) f'(t) \, dt. 
 \end{equation}

\subsection{Golomb's bounds for powerful numbers}
 A positive integer $n$ is said to be \emph{powerful} if $p \mid n \Longrightarrow p^2 \mid n$ for an arbitrary prime $p$. Equivalently, 
 we have that $n = b(m)$ for some $m$, according to the definition in \eqref{displaybn}. Being consistent with notation from Golomb 
 \cite{Golomb1970}, we write $k(x)$ in place of the number of powerful numbers $\leq x$. A remarkable result due to Golomb is such that 
\begin{equation}\label{displayGolomb}
 c_4 x^{1/2} - 3 x^{1/3} \leq k(x) \leq c_4 x^{1/2} 
\end{equation}
 for the constant $c_4 = \frac{\zeta\left( \frac{3}{2} \right)}{\zeta(3)}$. We apply this result from Golomb in Section \ref{secmain} below. 

\subsection{Tur\'an's proof of the Hardy--Ramanujan theorem}\label{secTuran}
 The \emph{Tur\'an--Kubilius inequality} provides a powerful tool for studying the average behavior of additive arithmetic functions 
 \cite[\S3.2]{Tenenbaum2015}. Being consistent with notation from Tenenbaum's text on analytic and probabilistic number theory 
 \cite[\S3.2]{Tenenbaum2015}, for an arithmetic function $f$, we denote the associated mean/expectation so that 
\begin{equation}\label{eq3p2}
 g(N) = \mathbb{E}_{N}(f) := \frac{1}{N} \sum_{1 \leq n \leq N} f(n). 
\end{equation} 
 As suggested by Tenenbaum, one would hope, depending on the input function $f$, that the mean in \eqref{eq3p2} would provide a 
 reasonable approximation to a normal order of $f$. Letting the associated variance be written as 
\begin{equation}\label{eq3p4}
 \mathbb{V}_{N}(f) := \mathbb{E}_{N}\big( |f(n) - g(N)|^2 \big), 
\end{equation} 
 one would want to obtain a relation involving \eqref{eq3p2} and \eqref{eq3p4} to guarantee or formalize how $g$ is, as desired, a 
 normal order for $f$. 
 This provides a way, for the $f(n) = \omega(n)$ case, 
 of proving that 
\begin{equation}\label{displayTuran}
 \sum_{n \leq N} \big( \omega(n) - \log \log N \big)^{2} \ll N \log \log N, 
\end{equation}
 with \eqref{displayTuran} having been proved in a non-probabilistic and elementary way by Tur\'an in 1934 \cite{Turan1934} to prove the 
 Hardy--Ramanujan theorem, which concerns the normal order of $\omega(n)$. Tur\'an's result in \eqref{displayTuran} 
 will be used in the derivation of the order-of-magnitude estimate in 
 Section \ref{secmain} below. 
 For the subsequent asymptotic refinement, we also use the fourth-moment estimate 
\begin{equation}\label{GSestimate} 
 \sum_{n \leq N} \big( \omega(n) - \log \log N \big)^{4} 
 \ll N (\log \log N)^2 
\end{equation}
 for $N \to \infty$, 
 which follows from the $k = 4$ case of a result due to Granville and Soundararajan
 \cite[Theorem 1]{GranvilleSoundararajan2007}. 
 
\begin{remark}
 Theorem \ref{MainTheorem} follows from the stronger Theorem \ref{mainmain} below, 
 but, for the purposes of structuring this paper and organizing the preliminary results required 
 for Theorem \ref{mainmain}, we choose to prove Theorem \ref{MainTheorem} separately. 
\end{remark} 

\section{Main results}\label{secmain}

\begin{lemma}\label{lemmaB1}
 Letting $$ \mathcal{B}_{1}(x) := \left\{ n \leq x : b(n) > (\log \log x)^{2} \right\}, $$ we have that $$ \sum_{n \in \mathcal{B}_{1}(x)} 
 \frac{1}{d(d(n))} \ll \frac{x}{\log \log x}. $$
\end{lemma}

\begin{proof} 
 Write $Y = (\log\log x)^2$. For a fixed powerful number $\beta$, the number of positive integers $n \leq x$ such that $b(n) = \beta$ is 
 bounded above by the \emph{total} number $\lfloor \frac{x}{\beta} \rfloor$ of positive multiples of $\beta$ not exceeding $x$, and 
 hence the relations whereby 
\begin{equation}\label{firstA2proof}
\#\mathcal{B}_1(x) \leq \sum_{\substack{\beta > Y\\ \beta \ {\rm powerful}}} \#\{n\le x : b(n) = \beta\} \leq 
 x \sum_{\substack{\beta > Y\\ \beta \ {\rm powerful}}}\frac1\beta.
\end{equation}
 By the Golomb estimate in \eqref{displayGolomb}, we find that 
\begin{equation}\label{consequenceGolomb}
k(t)\ll t^{1/2}.
\end{equation}
 In the formulation of the Abel summation formula given in Section \ref{secprelim}, we set $a_n = 1$ if $n$ is powerful and $a_n = 
 0$ otherwise, and we set $f(t) = \frac{1}{t}$, so that $$ \sum_{\substack{Y < \beta \leq Z\\ \beta\ {\rm powerful}}} \frac{1}{\beta} = \frac{k(Z)}{Z} - 
 \frac{k(Y)}{Y} + \int_Y^Z \frac{k(t)}{t^2} \, dt, $$ and hence 
\begin{equation}\label{withoutY}
 \sum_{\substack{Y < \beta \leq Z\\ \beta\ {\rm powerful}}}\frac{1}{\beta} \leq \frac{k(Z)}{Z} + \int_Y^Z \frac{k(t)}{t^2}\,dt. 
\end{equation}
 Using the consequence in \eqref{consequenceGolomb} of the Golomb bounds, and letting \(Z\to\infty\) in \eqref{withoutY} and applying 
 the monotone convergence theorem, it follows that
\begin{equation}\label{aftermonotone}
 \sum_{\substack{\beta>Y\\ \beta\ {\rm powerful}}}\frac{1}{\beta} \ll Y^{-1/2}.
\end{equation} 
 This estimate holds for every \(Y\geq1\), with an implied constant independent of \(Y\).
 From \eqref{firstA2proof} and \eqref{aftermonotone} together, we find that $ \#\mathcal{B}_1(x)\ll \frac{x}{\log\log x}$. Finally, since 
 $d(d(n)) \geq 1$, we have $$ \sum_{n \in \mathcal{B}_1(x)} \frac1{d(d(n))} \leq \#\mathcal{B}_1(x) \ll \frac{x}{\log\log x}, $$ as desired. 
\end{proof}

\begin{lemma}\label{LemmaB2}
Let $$ \mathcal{B}_2(x) := \left\{n \leq x:\omega(n) \notin \left[\frac12\log\log x,2\log\log x\right]\right\}. $$ Then $$ \sum_{n \in 
 \mathcal{B}_2(x)} \frac{1}{d(d(n))} \ll \frac{x}{\log\log x}. $$
\end{lemma}

\begin{proof}
 If \(n\in\mathcal{B}_2(x)\), then
\begin{equation}\label{firstproofB2}
 \frac12\log\log x \leq \left|\omega(n)-\log\log x\right|. 
\end{equation}
 From \eqref{firstproofB2}, we obtain that
\begin{align*}
 \frac{1}{4} (\log \log x)^{2} \# \mathcal{B}_{2}(x) 
 & \leq \sum_{n \in \mathcal{B}_{2}(x)} \big( \omega(n) - \log \log x \big)^{2} \\ 
 & \leq \sum_{n \leq x} \big( \omega(n) - \log \log x \big)^{2} \\ 
 & \ll x \log \log x, 
\end{align*}
 where we have applied the Tur\'{a}n estimate reviewed in Section \ref{secTuran}. So, since \(0 < \frac{1}{d(d(n))} \leq 1\), we see that 
$$ \sum_{n\in\mathcal{B}_2(x)}
\frac1{d(d(n))}
\leq \#\mathcal{B}_2(x)
\ll \frac{x}{\log\log x}, $$
 as desired. 
\end{proof}

\begin{lemma}\label{LemmaComplement}
We have that 
$$ \sum_{\substack{n\leq x\\
n\notin \mathcal{B}_1(x) \cup \mathcal{B}_2(x)}}
\frac1{d(d(n))} \ll
\frac{x}{\log\log x}. $$
\end{lemma}

\begin{proof}
 Let \(n \le x\), with \(n \notin \mathcal{B}_1(x) \cup \mathcal{B}_2(x)\). Recalling the divisor function identity in \eqref{keyErdos}, and 
 following a similar approach as in the work of Erd\H{o}s \cite{Erdos1968}, we write $ d(b(n)) = 2^{u(n)} v(n)$ for $v(n)$ odd, giving 
 us that 
\begin{equation}\label{userelative}
 d(n) = 2^{\omega(a(n)) + u(n)} v(n).
\end{equation}
 Since $2^{\omega(a(n)) + u(n)}$ and $v(n)$ are relatively prime, we obtain from \eqref{userelative} that 
\begin{equation}\label{ddnwan}
 d(d(n)) = \big(\omega(a(n)) + 
 u(n) + 1\big) d(v(n)). 
\end{equation}
 This gives us the lower bound 
\begin{equation}\label{lowerboundddn}
d(d(n)) \ge \omega(a(n)) + 1.
\end{equation}
 We proceed to estimate \(\omega(a(n))\). Since \(n\notin \mathcal B_2(x)\), we have that 
\begin{equation}\label{consequencenotB2}
\omega(n)\geq \frac12\log\log x.
\end{equation}
 Moreover, since \(n\notin \mathcal B_1(x)\), we have that 
\begin{equation}\label{sincenotinB1}
 b(n) \leq (\log\log x)^2.
\end{equation}
 Since 
\begin{equation}\label{omegadiff}
 \omega(a(n))=\omega(n)-\omega(b(n)), 
\end{equation}
 and since \eqref{sincenotinB1} gives us that 
\begin{equation}\label{2toomegabn}
 2^{\omega(b(n))} \leq b(n) \leq (\log\log x)^2, 
\end{equation}
we obtain
\begin{equation}\label{omegalogloglog}
 \omega(b(n))\leq \frac{2\log\log\log x}{\log 2}.
\end{equation}
 A combined application of \eqref{consequencenotB2}, \eqref{omegadiff}, and \eqref{omegalogloglog} then allows us to obtain that 
\begin{equation}\label{omegaangeq}
 \omega(a(n)) \geq \frac12\log\log x - \frac{2\log\log\log x}{\log 2} \gg \log\log x.
\end{equation}
 In turn, combining \eqref{lowerboundddn} and \eqref{omegaangeq} allows us to deduce that 
\begin{equation}\label{notAbel}
 \frac{1}{d(d(n))} \ll \frac{1}{\log\log x}, 
\end{equation}
 uniformly for $n \leq x$ such that $n \not\in \mathcal{B}_{1}(x) \cup \mathcal{B}_{2}(x)$. Applying the summation operator 
 $\sum_{\substack{n \leq x \\ n \not\in \mathcal{B}_{1}(x) \cup \mathcal{B}_{2}(x) }} \cdot $ to both sides of \eqref{notAbel} then allows 
 us to obtain the desired result. 
\end{proof}

\begin{lemma}\label{LemmaLower}
 We have that 
\begin{equation}\label{displayLower} 
 \frac{x}{\log\log x} \ll \sum_{n\le x}\frac{1}{d(d(n))}. 
\end{equation}
\end{lemma}

\begin{proof}
 As pointed out to us by an anonymous referee, the lower bound in \eqref{displayLower}
 follows from
the Cauchy--Schwarz inequality and the Erd\H{o}s asymptotic
formula in \eqref{displayErdos}. Indeed, Cauchy--Schwarz gives
$$ \lfloor x\rfloor^2
\leq
\left(\sum_{n \leq x} d(d(n))\right)
\left(\sum_{n \leq x} \frac{1}{d(d(n))}\right). $$ Consequently, by \eqref{displayErdos}, the relations in 
$$ \sum_{n \leq x} \frac{1}{d(d(n))}
\geq
\frac{\lfloor x\rfloor^2}
{\displaystyle\sum_{n \leq x} d(d(n))}
\gg \frac{x}{\log\log x},
 $$ hold, giving us the desired result. 
\end{proof}

\begin{theorem}\label{MainTheorem}
We have
\begin{equation}\label{displaymain}
\sum_{n\le x}\frac{1}{d(d(n))}
\asymp
\frac{x}{\log\log x}.
\end{equation}
\end{theorem}

\begin{proof}
 The upper bound associated with the order of magnitude estimate in \eqref{displaymain} follows by decomposing the index set 
 $$ \{ n : n \leq x \} = \mathcal{B}_1(x) \cup \mathcal{B}_2(x) \cup \big( \{ n : n \leq x \} \setminus(\mathcal{B}_1(x) \cup 
 \mathcal{B}_2(x))\big) $$ and applying Lemmas \ref{lemmaB1}, \ref{LemmaB2}, 
 and \ref{LemmaComplement}. The corresponding lower bound is given in Lemma \ref{LemmaLower}.
\end{proof}

\begin{remark}
As also observed by an anonymous referee, the upper bound in
\eqref{displaymain} can be obtained without Golomb's estimate.
Indeed the definition of $b(n)$ gives us that 
\begin{align*}
\sum_{n \leq x} \omega(b(n))
 & = 
\sum_{p \leq \sqrt{x}} 
\left\lfloor \frac{x}{p^2} \right\rfloor \\ 
 & \leq x \sum_p \frac{1}{p^2} \\ 
 & \ll x.
\end{align*}
Consequently, $\omega(b(n)) \leq \frac14\log\log x$
outside a set of $O(x/\log\log x)$ integers $n \leq x$.
Also, \eqref{displayTuran} implies that
$\omega(n) \geq \frac12\log\log x$ outside another set
of $O(x/\log\log x)$ integers. Hence, by \eqref{omegadiff}, we have that  $$ \omega(a(n)) \geq \frac14\log\log x $$
outside a set of $O(x/\log\log x)$ integers. The inequality \eqref{lowerboundddn}, which holds for every
positive integer $n$, then gives
$1/d(d(n)) \leq 4/\log\log x$ outside this exceptional set.
Bounding each remaining summand by $1$ proves the upper bound.
We retain the powerful-number argument above because
the tail estimate \eqref{aftermonotone} is also useful
for the asymptotic refinement.
\end{remark}

\subsection{An exact asymptotic formula}
 The goal of this section is to improve upon Theorem \ref{MainTheorem}, 
 by deriving an exact asymptotic formula for $\sum_{n \leq x} \frac{1}{d(d(n))}$. 

\begin{lemma}\label{fourthlogbn}
 We have that 
 $$ \sum_{n \leq x} \left( \log b(n) \right)^{4} \ll x. $$
\end{lemma}

\begin{proof}
 For every powerful integer $\beta$, 
 the number of positive integers $n \leq x$ 
 such that $b(n) = \beta$ is at most $\frac{x}{\beta}$. 
 Grouping positive integers $\leq x$ according to their powerful parts, we find that 
\begin{align*}
 \sum_{n \leq x} \left( \log b(n) \right)^{4} 
 & = \sum_{\substack{\beta \leq x \\ \text{$\beta$ powerful}}} \left( \log \beta \right)^{4}
 \#\{ n \leq x : b(n) = \beta \} \\ 
 & \leq x \sum_{\substack{\beta \leq x \\ \text{$\beta$ powerful}}} \frac{(\log \beta)^4}{\beta}.
\end{align*}
 By rewriting the factor 
 $(\log \beta)^4$ as 
 $4 \int_{1}^{\beta} \frac{(\log t)^3}{t} \, dt$, and by applying the estimate in \eqref{aftermonotone}, 
 an interchange of finite summation and integration then gives us that 
\begin{align*}
 \sum_{\substack{\beta \leq x \\ \text{$\beta$ powerful}}} \frac{(\log \beta)^4}{\beta} 
 & = 4 \int_{1}^{x} \frac{(\log t)^3}{t} \sum_{\substack{t < \beta \leq x \\ \text{$\beta$ powerful}}} \frac{1}{\beta} \, dt \\
 & \ll \int_{1}^{x} \frac{(\log t)^3}{t^{3/2}} \, dt \\ 
 & \leq \int_{1}^{\infty} \frac{(\log t)^3}{t^{3/2}} \, dt < \infty. 
\end{align*}
 Since the final bound is independent of $x$, the desired result follows. 
\end{proof}

 We proceed to combine Lemma \ref{fourthlogbn}
 with the Granville--Soundararajan 
 estimate in \eqref{GSestimate}, and, informally, 
 this is to control the integers for which 
 $\omega(a(n))$ is unusually small. 

\begin{lemma}\label{numberofwan}
 As $x \to \infty$, we have that 
 $$ \#\left\{ n \leq x : \omega(a(n)) < \frac{1}{2} \log \log x \right\} 
 \ll \frac{x}{(\log \log x)^2}. $$
\end{lemma}

\begin{proof}
 Let $x$ be sufficiently large so that $\log \log x \geq 1$. From the first relation 
 on display in \eqref{2toomegabn}, we find that Lemma \ref{fourthlogbn} gives us that 
\begin{align*}
 \sum_{n \leq x} \omega^{4}(b(n))
 & \leq \frac{1}{\log^4 2} \sum_{n \leq x} \log^4 b(n) \\
 & \ll x.
\end{align*}

 The identity highlighted in \eqref{omegadiff} implies that 
\begin{equation}\label{omegaabs}
 \left| \omega(a(n)) - \log \log x \right| 
 \leq \left| \omega(n) - \log \log x \right| + \omega(b(n)). 
\end{equation}
 Using the elementary inequality $(s+t)^4 \leq 8 (s^4 + t^4)$ for $s, t \geq 0$ along with the Granville--Soundararajan estimate on 
 display in \eqref{GSestimate}, 
 we find, using \eqref{omegaabs}, that 
\begin{align*}
 \sum_{n \leq x} \left( \omega(a(n)) - \log \log x \right)^4 
 & \leq 8 \sum_{n \leq x} \left( \omega(n) - \log \log x \right)^{4} 
 + 8 \sum_{n \leq x} \omega^{4}(b(n)) \\ 
 & \ll x \, \log^2 \log x + x \\ 
 & \ll x \, \log^2 \log x.
\end{align*}
 Now, if $\omega(a(n)) < \frac{1}{2} \log \log x$, then
 $\big| \omega(a(n)) - \log \log x \big| \geq \frac{1}{2} \log \log x$. 
 We thus obtain that 
\begin{align*}
 & \frac{\log^4 \log x}{16} \#\left\{ n \leq x : \omega(a(n)) < \frac{1}{2} \log \log x \right\} \\ 
 & 
 \leq \sum_{\substack{n \leq x \\ \omega(a(n)) < \frac{1}{2} \log \log x } } 
 \left( \omega(a(n)) - \log \log x \right)^4 \\ 
 & \leq \sum_{n \leq x} \left( \omega(a(n)) - \log \log x \right)^{4} \\ 
 & \ll x \log^2 \log x, 
\end{align*}
 giving us an equivalent version of the desired result. 
\end{proof}

 For every positive integer $n$, write $$ d(b(n)) = 2^{u(n)} v(n), $$
 where $u(n)$ is a nonnegative integer and $v(n)$ is odd, using the notation introduced in the 
 proof of Lemma 
 \ref{LemmaComplement}. The identity 
 $$ d(d(n)) = \big( \omega(a(n)) + u(n) + 1 \big) d(v(n)) $$
 established in the proof of 
 Lemma 
 \ref{LemmaComplement} also holds for each positive integer $n$. 

\begin{lemma}\label{lemloglogdd}
 As $x \to \infty$, we have that 
\begin{equation}\label{displayloglogdd}
 \sum_{n \leq x} \left| \frac{\log \log x}{d(d(n))} 
 - \frac{1}{d(v(n))} \right| 
 \ll \frac{x}{\sqrt{\log \log x}}. 
\end{equation}
\end{lemma}

\begin{proof}
 We henceforth let $x$ be such that $\log \log x \geq 1$. 

 Applying the natural logarithm to both sides of $2^{\omega(b(n))} \leq b(n)$, which holds for every positive
 integer $n$, this gives 
\begin{equation}\label{omegabnleq} 
 \omega(b(n)) \leq \frac{\log b(n)}{\log 2}. 
\end{equation}

 From the definition of $u(n)$ and $v(n)$, as in the relation
 $d(b(n)) = 2^{u(n)} v(n)$, 
 since $d(m) \leq m$ for each positive integer $m$, 
 we find that 
\begin{equation}\label{fromunvn} 
 2^{u(n)} \leq d(b(n)) \leq b(n). 
\end{equation}
 From \eqref{fromunvn}, it is immediate that 
\begin{equation}\label{unleq} 
 u(n) \leq \frac{\log b(n)}{\log 2}. 
\end{equation}

 A combined application of H\"{o}lder's inequality and Lemma \ref{fourthlogbn} gives us that 
\begin{align}
\begin{split}
 \sum_{n \leq x} \log b(n)
 & \leq \lfloor x \rfloor^{3/4} \left( \sum_{n \leq x} \left( \log b(n) \right)^{4} \right)^{1/4} \\ 
 & \ll x^{3/4} x^{1/4} = x.
\end{split}\label{splitlogb}
\end{align}
 Applying \eqref{omegabnleq} and \eqref{unleq}, respectively, and 
 then applying \eqref{splitlogb}, 
 we obtain that 
\begin{equation}\label{sumomegabn} 
 \sum_{n \leq x} \omega(b(n)) \leq \frac{1}{\log 2} \sum_{n \leq x} \log b(n) \ll x 
\end{equation}
 and that 
\begin{equation}\label{sumun} 
 \sum_{n \leq x} u(n) \leq \frac{1}{\log 2} \sum_{n \leq x} \log b(n) \ll x. 
\end{equation}
 Now, by the Cauchy--Schwarz inequality and 
 the second moment formula in \eqref{displayTuran}, we see that 
\begin{align}
\begin{split}
 \sum_{n \leq x} |\omega(n) - \log \log x|
 & \leq \lfloor x \rfloor^{1/2} \left( \sum_{n \leq x} (\omega(n) - \log \log x)^{2} \right)^{1/2} \\ 
 & \ll x^{1/2} (x \log \log x)^{1/2} = x \sqrt{\log \log x}. 
\end{split}\label{deviationomega}
\end{align}
 Now, a combined application of 
 \eqref{omegadiff}, the triangle inequality, 
 \eqref{sumomegabn}, \eqref{sumun}, and \eqref{deviationomega} 
 gives us that 
\begin{align*}
 \sum_{n \leq x} & \left| \omega(a(n)) + u(n) + 1 - \log \log x \right| \\ 
 & = \sum_{n \leq x} \left| \omega(n) - \log \log x - \omega(b(n)) + u(n) + 1 \right| \\
 & \leq \sum_{n \leq x} \left| \omega(n) - \log \log x \right| 
 + \sum_{n \leq x} \omega(b(n)) + \sum_{n \leq x} u(n) + \lfloor x \rfloor \\ 
 & \ll x \sqrt{\log \log x} + x + x + x \\ 
 & \ll x \sqrt{\log \log x},
\end{align*}
 recalling that $\log \log x \geq 1$. 

 Now, recall the relation that is displayed in 
 \eqref{ddnwan} and that is involved in the proof of Lemma \ref{LemmaComplement}. 
 Informally, our strategy, at this point, is to split the sum in 
 \eqref{displayloglogdd} according to whether 
\begin{equation}\label{omegaacase1}
 \omega(a(n)) \geq \frac{1}{2} \log \log x
\end{equation}
 or 
\begin{equation}\label{reversecase} 
 \omega(a(n)) < \frac{1}{2} \log \log x.
\end{equation} 

 If \eqref{omegaacase1} holds, then \eqref{ddnwan} implies that 
\begin{align}
\begin{split}
 \left| \frac{\log \log x}{d(d(n))} - \frac{1}{d(v(n))} \right| 
 & = \frac{\left| \log \log x - \omega(a(n)) - u(n) - 1 \right|}{ 
 \big( \omega(a(n)) + u(n) + 1 \big) \, d(v(n)) } \\ 
 & \leq \frac{2}{\log \log x} \big| \omega(a(n)) + u(n) + 1 - \log \log x \big|, 
\end{split}\label{conditionsummand}
\end{align}
 where we have employed the relations such that 
 $u(n) \geq 0$ and $d(v(n)) \geq 1$. 
 Summing the left-hand side of the equality in \eqref{conditionsummand}
 and making use of the inequality in \eqref{conditionsummand}
 and by applying the bound on 
 $\sum_{n \leq x} \big| \omega(a(n)) + u(n) + 1 - \log \log x \big|$
 established above, we find that 
\begin{align*}
 \sum_{\substack{n \leq x \\ \omega(a(n)) \geq \frac{1}{2} \log \log x }} 
 & \left| \frac{\log \log x}{d(d(n))} - \frac{1}{d(v(n))} \right| \\ 
 & \leq \frac{2}{\log \log x} \sum_{n \leq x} \big| \omega(a(n)) + u(n) + 1 - \log \log x \big| \\ 
 & \ll \frac{x \sqrt{\log \log x}}{\log \log x} \\ 
 & = \frac{x}{\sqrt{\log \log x}}. 
\end{align*}

 Now, suppose that \eqref{reversecase} were to hold. 
 Since $d(d(n)) \geq 1$ and since $d(v(n)) \geq 1$, 
 each term of the form 
 $\frac{\log \log x}{d(d(n))} - \frac{1}{d(v(n))}$
 is at most $\log \log x + 1$ in absolute value. So, an application of 
 Lemma \ref{numberofwan} gives us that 
\begin{align*}
 \sum_{\substack{n \leq x \\ \omega(a(n)) < \frac{1}{2} \log \log x }}
 & \left| \frac{\log \log x}{d(d(n))} - \frac{1}{d(v(n))} \right| \\ 
 & \leq (\log \log x + 1) \#\left\{ n \leq x : \omega(a(n)) < \frac{1}{2} \log \log x \right\} \\
 & \ll (\log \log x + 1) \frac{x}{\log^2 \log x} \\ 
 & \ll \frac{x}{\log \log x}. 
\end{align*}
 Combining the sum estimates above, and again making use of 
 the relation $\log \log x \geq 1$, we find that 
\begin{align*}
 \sum_{n \leq x} 
 \left| \frac{\log \log x}{d(d(n))} - \frac{1}{d(v(n))} \right| 
 & \ll \frac{x}{\sqrt{\log \log x}} + \frac{x}{\log \log x} \\ 
 & \ll \frac{x}{\sqrt{\log \log x}}, 
\end{align*}
 as desired. 
\end{proof}

\begin{lemma}\label{fromBrown}
 For every fixed and powerful and positive integer $\beta$, we have that 
 $$\lim_{x \to \infty} \frac{1}{x} \#\{ n \leq x : b(n) = \beta \} 
 = \frac{1}{\zeta(2) \beta} \prod_{p \mid \beta} \frac{p}{p+1}.$$
 For the $\beta = 1$ case, the associated empty product is interpreted as $1$. 
\end{lemma}

\begin{proof}
 The definition of the powerful part of an integer gives us that 
 $b(n) = \beta$ if and only if 
 $n = \beta a$ for $a$ squarefree such that $(a, \beta) = 1$. As a consequence, we have that 
 $$ \#\{ n \leq x : b(n) = \beta \} = \#\left\{ a \leq \frac{x}{\beta} : \text{$a$ squarefree, $(a, \beta) = 1$} \right\}.$$
 By Theorem 1 from the work of Brown \cite{Brown2021}, 
 applied with $T = \varnothing$ and $P = \{ p : p \mid \beta \}$, we have that 
 $$ \lim_{y \to \infty} \frac{1}{y} \#\{ a \leq y : \text{$a$ squarefree, \ $(a, \beta) = 1$} \} 
 = \frac{1}{\zeta(2)} \prod_{p \mid \beta} \frac{p}{p+1}. $$
 By then letting $y = \frac{x}{\beta}$, we thus obtain the desired result. 
\end{proof}
 
\begin{lemma}\label{lastlemma}
 The series 
\begin{equation}\label{evaluatec3} 
 \frac{1}{\zeta(2)} \sum_{\substack{\beta \geq 1 \\ \text{$\beta$ $\operatorname{powerful}$}}} \frac{1}{\beta d(v(\beta))} 
 \prod_{p \mid \beta} \frac{p}{p+1} 
\end{equation}
 converges absolutely. 
 If $c_3$ denotes its sum, then
 $$ \lim_{x \to \infty} \frac{1}{x} \sum_{n \leq x} \frac{1}{d(v(n))} = c_{3}, $$ 
 and $\frac{6}{\pi^2} \leq c_3 \leq 1$. 
\end{lemma}

\begin{proof}
 If $\beta$ is powerful, then $b(\beta) = \beta$. In this case, we have that $d(\beta) = 2^{u(\beta)} v(\beta)$. Observe that $v(\beta)$ 
 (by definition) is the odd part of $d(\beta)$. The uniqueness of the factorization into a power of $2$ and an odd integer then allows us to 
 conclude that 
\begin{equation}\label{displayimply}
 b(n) = \beta \Longrightarrow v(n) = v(\beta). 
\end{equation}
 We begin by verifying that the series in \eqref{evaluatec3} converges. 
 Since $d(v(\beta)) \geq 1$, 
 and since each factor of the form $\frac{p}{p+1}$ is at most $1$, 
 we see that 
\begin{equation}\label{verifyconverge}
 0 < \frac{1}{\zeta(2) \beta \, d(v(\beta))} \prod_{p \mid \beta} \frac{p}{p+1} \leq \frac{1}{\zeta(2) \beta}. 
\end{equation}
 The relation in \eqref{aftermonotone}, as given within the proof of Lemma \ref{lemmaB1}, then gives us the 
 desired convergence. 

 Now, let $B \geq 1$ be fixed. Informally, our strategy, at this point, 
 is to split the sum $\sum_{n \leq x} \frac{1}{d(v(n))}$ according to whether or not $b(n) > B$. 

 According to the implication 
 in \eqref{displayimply}, by grouping integers such that $b(n) \leq B$ according to their powerful parts, 
 we obtain that 
\begin{align*}
 \frac{1}{x} \sum_{\substack{n \leq x \\ b(n) \leq B}} \frac{1}{d(v(n))}
 & = \sum_{\substack{\beta \leq B \\ \text{$\beta$ powerful}}} \frac{1}{d(v(\beta))}
 \frac{\#\{ n \leq x : b(n) = \beta \}}{x} \\ 
 & \longrightarrow \frac{1}{\zeta(2)} 
 \sum_{\substack{\beta \leq B \\ \text{$\beta$ powerful}}}
 \frac{1}{\beta d(v(\beta))} \prod_{p \mid \beta} \frac{p}{p+1}
\end{align*}
 as $x \to \infty$, according to Lemma \ref{fromBrown}. 
 The limiting operation applied above is valid since $B$ is fixed and since the associated sum contains only finitely
 many terms. 

 For the $b(n) > B$ case suggested above, we make use of the relation $\frac{1}{d(v(n))} \leq 1$, 
 and, as in the proof of Lemma \ref{lemmaB1}, 
 the number of integers $n \leq x$ satisfying $b(n) = \beta$ is at most $\frac{x}{\beta}$. 
 We thus obtain that
\begin{align*}
 0 & \leq \frac{1}{x} \sum_{\substack{n \leq x \\ b(n) > B}} \frac{1}{d(v(n))} \\
 & \leq \frac{1}{x} \sum_{\substack{\beta > B \\ \text{$\beta$ powerful}}} \#\{ n \leq x : b(n) = \beta \} \\ 
 & \leq \sum_{\substack{\beta > B \\ \text{$\beta$ powerful} }} \frac{1}{\beta} \\ 
 & \ll B^{-1/2}, 
\end{align*}
 where this last relation is implied by the tail estimate in \eqref{aftermonotone}, within the proof of Lemma \ref{lemmaB1}. From the 
 same tail estimate, in conjunction with \eqref{verifyconverge}, we find that 
\begin{align*}
 0 & \leq c_3 - \frac{1}{\zeta(2)} \sum_{\substack{\beta \leq B \\ \text{$\beta$ powerful}}} \frac{1}{\beta d(v(\beta))} 
 \prod_{p \mid \beta} \frac{p}{p + 1} \\ 
 & \leq \frac{1}{\zeta(2)} \sum_{\substack{\beta > B \\ \text{$\beta$ powerful}}} \frac{1}{\beta} \\ 
 & \ll B^{-1/2}.
\end{align*}
 A combined application of the three sequences of relations displayed above gives us that: For every fixed $B \geq 1$, the relation 
 $$\limsup_{x \to \infty} \left| \frac{1}{x} \sum_{n \leq x} \frac{1}{d(v(n))} - c_{3} \right| 
 \ll B^{-1/2} $$
 holds. The implied constants in these two tail bounds are independent of \(x\) and \(B\).
 Letting \(B\to\infty\) proves that the mean of \(1/d(v(n))\) converges to \(c_3\).

 Finally, the term in \eqref{evaluatec3} corresponding to $\beta = 1$
 equals $\frac{1}{\zeta(2)}$, and, since the remaining terms are positive, 
 we find that $c_{3} \geq \frac{1}{\zeta(2)}$. 
 On the other hand, we see that 
 $$ \frac{1}{x} \sum_{n \leq x} \frac{1}{d(v(n))} \leq \frac{\lfloor x \rfloor}{x}. $$
 By taking the limit, with $x \to \infty$, we find that $c_{3} \leq 1$. 
\end{proof}

\begin{theorem}\label{mainmain}
 The asymptotic relation $$\sum_{n \leq x} \frac{1}{d(d(n))} \sim c_{3} \frac{x}{\log \log x} $$ holds for the positive and finite constant 
 $c_{3}$ given explicitly in Lemma \ref{lastlemma}. 
\end{theorem}

\begin{proof}
 A combined application of the triangle inequality and Lemma \ref{lemloglogdd} gives us that 
\begin{align*}
 & \left| \frac{\log \log x}{x} \sum_{n \leq x} \frac{1}{d(d(n))} 
 - \frac{1}{x} \sum_{n \leq x} \frac{1}{d(v(n))} \right| \\ 
 & \leq \frac{1}{x} \sum_{n \leq x} \left| \frac{\log \log x}{d(d(n))} - \frac{1}{d(v(n))} \right| \\ 
 & \ll \frac{1}{\sqrt{\log \log x}}.
\end{align*}
 On the other hand, Lemma \ref{lastlemma} gives us that 
 $$ \frac{1}{x} \sum_{n \leq x} \frac{1}{d(v(n))} = c_3 + o(1) $$
 as $x \to \infty$. A combined application of the above relations gives us that 
\begin{align*} 
 \frac{\log \log x}{x} \sum_{n \leq x} \frac{1}{d(d(n))} 
 & = c_3 + o(1) + O\left( \frac{1}{\sqrt{\log \log x}} \right) \\ 
 & = c_3 + o(1), 
\end{align*} 
 giving us the desired result. 
\end{proof}

\section{Conclusion}
 It was conjectured by Bellman and Shapiro \cite{BellmanShapiro1948}      that: For each fixed integer $r \geq 1$, there exists a constant 
    $\mathcal{C}_{r} \in (0, \infty)$ such that 
\begin{equation}\label{BellShapConj} 
 \sum_{n \leq x} \underbrace{d(d(\cdots d(}_{r}n))) \sim \mathcal{C}_{r} x \underbrace{\log(\log(\cdots \log(}_{r}x))). 
 \end{equation}
 This conjecture remains open for $r \geq 5$ \cite{KataiPhong2024}. 
 This and our results and techniques above 
 lead us to conjecture that: For each fixed integer $r \geq 2$, there exists a constant $\mathcal{D}_{r} \in (0, \infty)$
 such that 
\begin{equation}\label{BellShapRecip} 
 \sum_{n \leq x} \frac{1}{\underbrace{d(d(\cdots d(}_{r}n))) } \sim \mathcal{D}_{r} 
 \frac{x}{ \underbrace{\log(\log(\cdots \log(}_{r}x)))}. 
\end{equation}
 The $r = 2$ case is established in Theorem \ref{mainmain}, 
 with $\mathcal{D}_{2} = c_{3}$; the cases $r \geq 3$ are proposed here as conjectural extensions. 

 Observe that \eqref{BellShapConj} \emph{does} hold for $r=1$, by Dirichlet's formula in \eqref{displayDirichlet}, whereas the 
 corresponding reciprocal formula in \eqref{BellShapRecip} would \emph{not} hold if extended to $r=1$, by Ramanujan's formula in 
 \eqref{displayRamanujan}. Indeed, the latter gives
an asymptotic scale of $x/\sqrt{\log x}$ rather than
$x/\log x$. This illustrates that taking reciprocals does
not simply invert the logarithmic factor in the mean-value
formula.

\subsection*{Acknowledgements}
 The author is very grateful to an anonymous referee corresponding to this research paper,  who provided a hugely helpful report that led 
 to numerous improvements to this paper and that 
 helped the author to develop a proof of 
 the above asymptotic formula for 
 $\sum_{n \leq x} \frac{1}{d(d(n))}$, 
 and this was also greatly helped through the author's extensive interactions with GPT-6 Astra Pro, 
 which also helped the author to prepare revisions for this research paper. 
 The author is also very thankful to an anonymous referee for a previous, unpublished version of this work for a number of very helpful 
 suggestions, which provided a basis for 
 much of the author's techniques concerning 
 decomposition of integers into squarefree and powerful parts and the above treatment of exceptional sets. 
 The author is also very thankful to Karl Dilcher for many useful discussions related to 
 this paper and thanks Marco Cantarini and Amiram Eldar for useful commentary related to this paper. 
 The author used the GPT-5.3 model to be of assistance with exposition and to explore heuristic arguments, but 
 all mathematical results and proofs were verified independently by the author.

\subsection*{Statements and Declarations}
 There are no financial or non-financial competing interests that are directly or indirectly related to the work 
 to disclose. 

 \bibliographystyle{plain}
\bibliography{seprefe}

\end{document}